\documentclass{amsart}

\usepackage{mathrsfs,amssymb,multirow,setspace,color}
\usepackage{hyperref}
\usepackage[a4paper, total={7in, 9in}]{geometry}
\theoremstyle{plain}
\usepackage{graphicx}
\newtheorem{theorem}{Theorem}[section]

\newtheorem{proposition}[theorem]{Proposition}
\newtheorem{corollary}[theorem]{Corollary}

\theoremstyle{definition}

\theoremstyle{remark}
\newtheorem{remark}[theorem]{Remark}

\input xy
\xyoption{all}

\def\pgh{\mathcal{P}_E({G})}

\def\sg{\mathcal{S}(G)}
\begin{document}
\title{On the minimal (edge) connectivity of graphs and its applications to power graphs of finite groups}

\author[Parveen, Manisha, Jitender Kumar]{Parveen, Manisha, Jitender Kumar$^{*}$}
 \address{$\text{}^1$Department of Mathematics, Birla Institute of Technology and Science Pilani, Pilani-333031, India}
\email{p.parveenkumar144@gmail.com, yadavmanisha2611@gmail.com, jitenderarora09@gmail.com}

\begin{abstract} 
In an earlier work, finite groups whose power graphs are minimally edge connected have been classified. In this article, first we obtain a necessary and sufficient condition for an arbitrary graph to be minimally edge connected. Consequently, we characterize finite groups whose enhanced power graphs and order superpower graphs, respectively, are minimally edge connected. Moreover, for a finite non-cyclic group $G$, we prove that $G$ is an elementary abelian $2$-group if and only if its enhanced power graph is minimally connected. Also, we show that $G$ is a finite $p$-group if and only if its order superpower graph is minimally connected.  Finally, we characterize all the finite nilpotent groups such that the minimum degree and the vertex connectivity of their order superpower graphs are equal.
 \end{abstract}

\subjclass[2020]{05C25}

\keywords{ edge connectivity, vertex connectivity, minimum degree, nilpotent groups. \\ *  Corresponding author}

\maketitle
\section{Historical Background and Preliminaries}
There are a number of graphs associated with finite groups, including the Cayley graph, power graph and the commuting graph. Such graphs gives some more insights between the algebraic properties of groups and graph theoretic properties of associated graphs. The use of algebraic techniques can give new constructions and analysis of graphs with crucial properties. Moreover,  Cayley graph has valuable applications in data mining and automata theory (see, \cite{a.kelarev2009cayley,kelarev2004labelled}). Hayat \emph{et al.} \cite{a.hayat2019novel} utilized commuting graphs to establish some Non-Singular with a Singular Deck (NSSD) molecular graphs. Due to these applications and connections between algebra and graph theory, certain subgraphs of commuting graphs, namely: power graphs and enhanced power graphs became a topic of discussion among researchers. 

The \emph{power graph} $\mathcal{P}(G)$ of a group $G$ is a simple undirected graph whose vertex set is $G$ and two vertices $x$ and $y$ are adjacent if either $x = y^m$ or $y = x^n$ for some $m, n \in \mathbb{N}$.  Various aspects of power graphs have been studied in the literature. Panda \cite{PANDA20201} proved that a finite group $G$ is a non-cyclic group of prime exponent if and only if the power graph $\mathcal{P}(G)$ is both non-complete and minimally edge-connected. Also, it was shown that a finite group $G$ is an elementary abelian $2$-group of rank at least $2$ if and only if $\mathcal{P}(G)$ is non-complete and minimally connected. Panda and Krishna \rm\cite{prasad2018minimum} classified finite nilpotent group $G$ for which the minimum degree and the vertex connectivity of $\mathcal{P}(G)$ are equal. For more details on the power graphs, we refer the readers to the survey paper \cite{powersurvey}. Then Aalipour \emph{et al.} \cite{a.Cameron2016}  introduced a new graph so-called the enhanced power graph. The \emph{enhanced power graph} $\mathcal{P}_E(G)$ of a group $G$ is a simple undirected graph whose vertex set is $G$, and two distinct vertices $x$ and $y$ are adjacent if $x, y \in \langle z \rangle$ for some $z \in G$.  Recently, the enhanced power graph has received significant attention from various researchers. Kumar \emph{et al.} \cite{a.parveen2022} described the groups such that the minimum degree and the vertex connectivity of their associated enhanced power graphs are equal. Also, they characterized finite groups whose (proper) enhanced power graphs are (strongly) regular. For more results on the enhanced power graphs, the reader may refer to the survey paper \cite{enhancedsurvey} and references therein. Another graph which contains power graph is the order superpower graph which was introduced by Hamzeh and Ashrafi \cite{a.Hamzeg2018}. The \emph{order superpower graph} $\mathcal{S}(G)$ of a group $G$ is a simple undirected graph with vertex set $G$ and two distinct vertices are adjacent if and only if the order of one vertex divides the order of the other. Together with the graph theoretic properties of $\mathcal{S}(G)$, Hamzeh and Ashrafi \cite{a.Hamzeg2018} have explored the relationship between the order superpower graph and the power graph. Notice that for a finite group $G$, the power graph $\mathcal{P}(G)$ is a spanning subgraph of the enhanced power graph $\mathcal{P}_E(G)$ and the order superpower graph $\mathcal{S}(G)$, respectively. The automorphism groups and full automorphism groups of the order superpower graph have been determined in \cite{HAMZEH201782} . Ma and Su \cite{ma2022order} investigated the independence number of the order superpower graph \( \mathcal{S}(G) \). Kumar et al. \cite{kumar2024superpower} established tight bounds for the vertex connectivity and studied Hamiltonian-like properties of the superpower graph for finite non-abelian groups possessing an element of exponent order.

This motivates the present paper, which considers the problem of classification of finite groups such that their enhanced power and superpower graphs, respectively, are minimally edge connected. In addition to that, we classify all the finite groups whose enhanced power graphs and order superpower graphs, respectively, are minimally connected. Finally, we determine all the finite nilpotent groups such that the minimum degree and the vertex connectivity of their order superpower graphs are equal.

Now, we briefly recall some notations and terminology that we used in the paper. A graph \emph{$\Gamma$} is a pair $\Gamma = (V,E)$, where $V(\Gamma)$ and $E(\Gamma)$ are the set of vertices and edges of $\Gamma$, respectively. Two distinct vertices $u$ and $v$ are adjacent, denoted by $u \sim v$, if there exists an edge $\{uv\}$ between $u$ and $v$. Multiple edges are the edges having the same endpoints. An edge $\{uv\}$ is called \emph{loop} if $u=v$. Simple graphs are the graphs with no loops or multiple edges. The set $N(x)$ of all the vertices which are adjacent to the vertex $x$ in $\Gamma$ is called the \emph{neighbourhood} of $x$. Furthermore, we denote $N[x] = N(x) \cup \{x\}$. A subgraph $\Gamma^{'}$ of a graph $\Gamma$ is a graph such that $V(\Gamma^{'}) \subseteq V(\Gamma)$ and $E(\Gamma^{'}) \subseteq E(\Gamma)$. A graph $\Gamma$ is said to be \emph{complete} if any two distinct vertices are adjacent. A \emph{walk} $\lambda$ in $\Gamma$ from vertex $u$ to the vertex $w$ is a sequence of vertices $u=v_1,v_2,\ldots,v_m = w(m>1)$ such that $v_i \sim v_{i+1}$ for every $i \in \{1,2,\ldots,m-1\}$. A walk is said to be \emph{path} if no vertex is repeated. A graph $\Gamma$ is connected if each pair of vertices has a path in $\Gamma$. Otherwise, $\Gamma$ is disconnected. The number of vertices adjacent to the vertex $u$ is called the \emph{degree} of $u$ and is denoted as $deg(u)$. A graph is called \emph{regular} if each of its vertex has same degree. A vertex $u$ is said to be a \emph{dominating} vertex of a graph $\Gamma$ if $u$ is adjacent to all the other vertices of $\Gamma$. The maximum distance between two vertices of a connected graph $\Gamma$ is called the \emph{diameter} of $\Gamma$ and is denoted by $diam(\Gamma)$. The \emph{minimum degree} $\delta(\Gamma)$ of a graph $\Gamma$ is the smallest degree among all the vertices of $\Gamma$. A \emph{vertex (edge) cut-set} in a connected graph $\Gamma$ is a set $S$ of vertices (edges) such that the remaining subgraph $\Gamma - S$, by removing the set $S$, is disconnected or has only one vertex. The \emph{vertex connectivity} $\kappa(\Gamma)$(\emph{edge connectivity} $\kappa'(\Gamma)$) of a connected graph $\Gamma$ is the minimum size of a vertex (edge) cut-set. A graph is said to be \emph{minimally edge-connected} if for any edge $\epsilon$ of graph $\Gamma$, $\kappa'(\Gamma - \epsilon) = \kappa'(\Gamma)-1$. 
A graph is said to be \emph{minimally connected} if for any edge $\epsilon$ of graph $\Gamma$, $\kappa(\Gamma - \epsilon) = \kappa(\Gamma)-1$. The following results are useful in the sequel.

\begin{theorem}[{\cite[Theorem 6]{MR0398904}}]\label{ch1-edgecon.mindeg}
If the diameter of any graph is at most $ 2 $, then its edge connectivity and minimum degree are equal.
\end{theorem}

\begin{theorem}[{\cite[Theorem 4.1.9]{b.westgraph}}] \label{ch1-vertex-edge-conectivity}
If $\Gamma$ is a simple graph, then
\[\kappa(\Gamma) \le \kappa'(\Gamma) \le \delta(\Gamma).\]
\end{theorem}

Let $G$ be a group and $p$ be a prime number. The \emph{order} $o(a)$ \emph{of an element} $a$ in $G$ is the cardinality of the subgroup generated by $a$. If $|G|=p^n$ for some prime $p$, then $G$ is called a \emph{$p$-group}. The set of orders of all the elements of $G$ is denoted by $\pi_G$. The \emph{exponent} of $G$ is defined as the least common multiple of the orders of all elements of $G$ and it is denoted by ${exp}(G)$. If $G$ contains an element whose order is equal to ${exp}(G)$, then $G$ is called the \emph{group of full exponent}.
  A cyclic subgroup of a group $G$ is called a \emph{maximal cyclic subgroup} if it is not properly contained in any cyclic subgroup of $G$. The set of all maximal cyclic subgroups of $G$ is denoted by $\mathcal{M}(G)$. For $n \geq 3$, the \emph{dihedral group} ${D}_{2n}$ of order $2n$ is defined as ${D}_{2n} = \langle x, y  :  x^{n} = y^2 = e,  xy = yx^{-1} \rangle $. Note that the group $D_{2n}$ has one maximal cyclic subgroup $M= \langle x \rangle$ of order $n$ and $n$ maximal cyclic subgroups $M_i = \langle x^iy\rangle$, where $1\le i \le n$, of order $2$. 
For $n \geq 2$, the \emph{generalized quaternion group} ${Q}_{4n}$ of order $4n$ is defined as ${Q}_{4n} = \langle a, b  :  a^{2n}  = e, a^n= b^2, ab = ba^{-1} \rangle$.  Observe that the group ${Q}_{4n}$ has one maximal cyclic subgroup $M=\langle a \rangle$ of order $2n$ and $n$ maximal cyclic subgroups $M_i = \langle a^ib\rangle$, where $1\le i \le n$, of order $4$. 
For $n \geq 2$, the \emph{semidihedral group} $SD_{8n}$ of order $8n$ is defined as $SD_{8n} = \langle a, b  :  a^{4n} = b^2 = e,  ba = a^{2n -1}b \rangle$. 
Note that the group $SD_{8n}$ has one maximal cyclic subgroup $M = \langle x\rangle$ of order $4n$, $2n$ maximal cyclic subgroups  $M_i = \langle a^{2i}b \rangle = \{e, a^{2i}b\}$, where $1\leq i \leq 2n$, of order $2$ and $n$ maximal cyclic subgroups $ M_j =  \langle a^{2j + 1}b \rangle = \{e, a^{2n}, a^{2j +1}b, a^{2n + 2j +1}b\} $, where $1\leq j \leq n$, of order $4$. The group of all permutations from a set of $n$ elements to itself is denoted by $S_n$.

 \begin{theorem}[{\rm \cite[p. 193]{b.dummit1991abstract}}]{\label{nilpotent}} Let $G$ be a finite group. Then the following statements are equivalent:
\begin{enumerate}
    \item[(i)] G is a nilpotent group.
    \item[(ii)] Every Sylow subgroup of $G$ is normal.
    \item[(iii)] $G$ is direct product of its Sylow subgroups.
    \item[(iv)] For $x,y \in G$, $x$ and $y$ commute whenever $o(x)$ and $o(y)$ are relatively prime.
\end{enumerate}
    \end{theorem}

\begin{theorem}[{\rm\cite[Theorem 4]{powerregularity}}] \label{power graph regularity} Let $G $ be a finite group. The reduced power graph $\mathcal{P}^*(G)$ is regular if and only if $G$ is isomorphic to the cyclic $ p$-group or $exp(G) = p$, where $p$ is prime.
\end{theorem}

 \begin{remark}\label{dominating enhanced}
     Let $G$ be a finite group. Then $x$ is a dominating vertex of the enhanced power graph $\pgh$ if and only if $x$ lies in every maximal cyclic subgroup of $G$.
 \end{remark}

\begin{theorem}[{\rm\cite[Theorem 4.1]{a.parveen2022}}]\label{regularity enhanced} Let $G$ be a finite group. Then $\mathcal{P}^*_E(G)$ is regular if and only if one of the following holds:
\begin{enumerate}
    \item [(i)]$G$ is a cyclic group.
    \item [(ii)]$|M_i|=|M_j|$ and $M_i \cap M_j= \{e\}$, where $M_i,M_j\in \mathcal{M}(G)$.
\end{enumerate}
    
\end{theorem}

\begin{theorem}[{\rm\cite[Theorem 4.5]{a.parveen2022}}]\label{nilpotent regularity}
    Let $G$ be a non-cyclic nilpotent group. Then $\mathcal{P}^*_E(G)$ is regular if and only if $G$ is a $p$-group with exponent $p$.
\end{theorem}

\section{Main Results}
The main results of the mansuscript are presented in this section. We classify finite groups such that their enhanced power graphs and order superpower graphs, respectively, are minimally edge connected. Also, we classify all the finite groups such that their enhanced power graphs (cf. Theorem \ref{pgh minimally connected}) and order superpower graphs (cf. Theorem \ref{sg-minimallyconnected}) are minimally connected. Then we determine all the finite nilpotent groups such that the minimum degree and the vertex connectivity of their order superpower graphs are equal (cf. Theorem \ref{min degree = vertex connectivity of SG}). We begin with the following theorem which provides a necessary and sufficient condition for an arbitrary graph $\Gamma$ having dominating vertices to be minimally edge connected.  
\begin{theorem}{\label{graph}}
    Let $\Gamma$ be a non-complete graph with a dominating vertex $x$. Then $\Gamma $ is minimally edge connected if and only if $x$ is the only dominating vertex in $\Gamma $ and $\Gamma -\{x\}$ is a regular graph.
\end{theorem}
\begin{proof}
    Suppose $\Gamma$ is minimally edge connected. Since $diam(\Gamma) = 2$, we get $\kappa'(\Gamma) = \delta(\Gamma) = deg(z)$, for some $z \in V(\Gamma)$. If possible, assume that $\Gamma $ has two dominating vertices $x$ and $y$. Consider the edge $\epsilon =\{xy\}$. Notice that $deg(x)= deg(y) > deg(z)$. After removing the edge $\epsilon$, we obtain that $deg(x)=deg(y) \geq deg(z)$. It follows that $\delta (\Gamma -\epsilon) = deg(z) = \delta (\Gamma)$. Note that $diam(\Gamma -\epsilon) =2$. Consequently, we have $\kappa'(\Gamma -\epsilon) = \delta(\Gamma - \epsilon) = \delta(\Gamma)$, a contradiction. Thus, $x$ is the only dominating vertex in $\Gamma$. Now suppose that the graph $\Gamma -\{x\}$ is not regular. Then there exist two vertices $y,z \in V(\Gamma)$ such that $deg(z) > deg(y)$. Consider the edge $\epsilon_1 =\{xz\}$ in $\Gamma$. Let $S$ be an edge cut-set of minimum size in $\Gamma -\epsilon_1$. Clearly, $|S| = \kappa'(\Gamma -\epsilon_1)$. Let $y_1$ and $y_2$ be two vertices in $(\Gamma -\epsilon_1)-S$ such that there exists no path between $y_1$ and $y_2$ in $(\Gamma -\epsilon_1)-S$. Since $\Gamma $ is minimally edge connected, we get $\kappa'(\Gamma) = |S|+1$. It follows that there exists a path $P_1 : y_1 \sim x_1 \sim \cdots \sim y_2$ from $y_1$ to $y_2$ in $\Gamma -S$. Observe that the edge $\{x z\}$ must belong to the path $P_1$. Otherwise, $P_1$ is a path between $y_1$ and $y_2$ in $(\Gamma - \epsilon_1) -S $, which is not possible. Consider the set $S'= \{ \{az\}: a\in V(\Gamma)-\{x\}\}$ of edges. Now, we claim that $|S'| \leq |S|$. Let $u \neq x \in N(z)$. Then by replacing $z \sim x$ by $z \sim u \sim x$ in the path $P_1$, we get a new path $P_2$ from $y_1$ to $y_2$ in $\Gamma- \epsilon_1$. Thus, corresponding to each vertex $u$ in  $N(z) \setminus \{x\}$, either $\{uz\}$ or $\{ ux\}$ must belong to $S$. Otherwise, $P_2$
 is a path between $y_1$ and $y_2$ in $(\Gamma -\epsilon )-S$, which is not possible. Thus, either $\{uz\}$ or $\{ux\}$ belongs to $S$. It implies that $|S'| \leq |S|$. Thus, $|S| \geq deg(z) -1 \geq deg(y) = \kappa'(\Gamma)$, which is not possible. Therefore, the graph $\Gamma -\{x\}$ is a regular.
 
Conversely, suppose that $x$ is the dominating vertex of $\Gamma$ and $\Gamma -\{x\}$ is a regular graph. Let $\epsilon =\{ y_1y_2\}$ be any edge in graph $\Gamma$. Now we have the following cases:\\
\textbf{Case-1:} $x \neq y_1,y_2$. In this case, $diam(\Gamma) = diam(\Gamma - \epsilon) =2$. Also, $\delta(\Gamma) = deg(y_1)$ and $\delta (\Gamma -\epsilon) =deg(y_1) -1$. It follows that $\kappa'(\Gamma) = \kappa'(\Gamma -\epsilon)+1$. Thus, $\Gamma$ is minimally edge connected.\\
\textbf{Case-2:} Either $y_1=x$ or $y_2 = x$. With no loss of generality, assume that $y_1 = x$. Then $\delta(\Gamma -\epsilon) = deg (y_2)-1$. We will show that $\kappa'(\Gamma -\epsilon) = deg(y_2) -1$. Let $S$ be an edge cut-set of minimum size in $\Gamma -\epsilon$. If possible, assume that $|S| < deg(y_2) -1$. Consider two vertices $z_1$ and $z_2$ such that there does not exist any path between $z_1$ and $z_2$ in $(\Gamma - \epsilon) -S$. Since $diam(\Gamma)= 2$, we have $\kappa'(\Gamma) = \delta( \Gamma) = deg(y_2)$. Thus, the graph $\Gamma -S$ is connected. Consequently, there exists a path $P: z_1 \sim u_1 \sim u_2 \sim \cdots \sim z_2$ between $z_1$ and $z_2$ in $\Gamma -S$. Observe that the edge $\{x y_2\}$ must belong to the path $P$. Otherwise, $P$ will become a path between $z_1$ and $z_2$ in $(\Gamma -\epsilon)-S $. For any $u \neq x \in N(y_2)$, we get a path $P_2: z_1 \sim u_1\sim u_2 \sim \cdots \sim x \sim u \sim y_2\sim \cdots \sim z_2$ in $\Gamma -\epsilon$. Thus, either the edge $\{xu\}$ or $\{uy_2\}$ belongs to $S$. It follows that $|S| \geq deg(y_2)-1$. Therefore, $|S| = deg(y_2) -1 = \kappa'(\Gamma -\epsilon)$. Hence, $\Gamma$ is minimally edge connected.
 \end{proof}

\begin{corollary}[{\rm\cite[Theorem 1.1]{PANDA20201}}]
    A finite group $G$ is a non-cyclic group of prime exponent if and only if the power graph $\mathcal{P}(G)$ is non-complete and minimally edge-connected.
\end{corollary}
\begin{proof}
The result follows from Theorems \ref{power graph regularity}, \ref{graph} and, {\rm\cite[Theorem 4]{camerondominatingpower}}.    
\end{proof}

\begin{corollary}\label{Minimally edge connected}
    Let $G$ be a finite non-cyclic group. Then the enhanced power graph $\pgh$ is minimally edge connected if and only if $|M_{i}|=|M_{j}|$ and $M_{i} \cap M_{j} =\{e\}$, where $M_{i},M_{j} \in \mathcal{M}(G).$
    
\end{corollary}

\begin{proof}
    The result holds by Remark \ref{dominating enhanced}, Theorem \ref{regularity enhanced} and Theorem \ref{graph}.
\end{proof}

\begin{corollary}
    Let $G$ be a finite non-cyclic nilpotent group. Then the enhanced power graph $\pgh$ is minimally edge-connected if and only if $G$ is a $p$-group with exponent $p$.
\end{corollary}
\begin{proof}
   On combining Corollary \ref{Minimally edge connected} with Theorem \ref{nilpotent regularity}, the result holds. 
\end{proof}

\begin{corollary}
    If $G \in \{ D_{2n}, Q_{4n}, SD_{8n} \}$, then the enhanced power graph $\pgh$ is not minimally edge-connected.
\end{corollary}

\begin{proof}
    For each of these groups, note that there exists two maximal cyclic subgroups $M_i$ and $M_j$ such that $|M_i| \neq |M_j| $. By Corollary \ref{Minimally edge connected}, the result holds.
\end{proof}

In the following theorem, we classify all the finite groups whose enhanced power graphs are minimally connected.
\begin{theorem}\label{pgh minimally connected}
    Let $G$ be a finite group. Then the enhanced power graph $\pgh$ is minimally connected if and only if one of the following holds:
    \begin{enumerate}
        \item[(i)] $G$ is cyclic.
        \item[(ii)] $G \cong \mathbb{Z}_{2} \times \mathbb{Z}_{2} \times \cdots \times \mathbb{Z}_{2}$.
    \end{enumerate}
\end{theorem}
\begin{proof}
    Suppose $\pgh$ is minimally connected. Assume that $G$ has an element $x$ such that $o(x) \geq 3$. Let $\epsilon = \{xx^{-1}\}$ and $S$ be a vertex cut-set of minimum size in the graph $\pgh - \epsilon$. Thus, we have $\kappa(\pgh -\epsilon ) =|S|$. Since $\pgh$ is minimally connected, we obtain $\kappa(\pgh) = |S|+1$. Consequently, there exist two elements $y_{1}$ and $y_{2}$ of $G$ such that there does not exist any path between them in $(\pgh-\epsilon)-S$.\\
    \textbf{Case-1:} \emph{$y_{i} \notin \{x,x^{-1}\}$ for at least one $i$}.
Since $\kappa(\pgh) = |S|+1$, the graph $\pgh - S$ is connected. Therefore, in $\pgh -S$ there exists a path $P_{1}$: $y_{1} \sim x_{1}\sim x_{2}\sim \cdots \sim x_{k}\sim x\sim x^{-1} \sim \cdots \sim y_{2}$ (containing the edge $\epsilon$) between $y_{1}$ and $y_{2}$. Since $x_{k} \sim x$,  we get $x_{k} \sim x^{-1}$. Consequently, we have another path $P_{2}: y_{1} \sim x_{1} \sim x_{2} \sim \cdots \sim x_{k} \sim x^{-1} \sim \cdots \sim x_{t} \sim y_{2}$ in $(\pgh - \epsilon) - S$, which is not possible.\\
    \textbf{Case-2:} \emph{$y_i \in \{x,x^{-1}\}$ for $1 \leq i \leq 2$.}
If $G = S \cup \{x, x^{-1}\}$, then $|S| = |G|-2$. Therefore, $\kappa( \pgh) = |G| - 1$. It follows that $\pgh$ is complete and so $G$ is cyclic (cf. {\rm\cite[Theorem 2.4]{a.Bera2017}}). We may now suppose that $G \neq S \cup \{x, x^{-1}\}$.
    In this case $x$ and $x^{-1}$ are isolated vertices in $(\pgh -\epsilon) - S$. Otherwise, if $x \sim y$ for some $y \in G\setminus S$, then $y \sim x^{-1}$. Consequently, we get a path $x \sim y \sim x
    ^{-1}$, a contradiction. Thus $x$ and $x^{-1}$ are isolated vertices in $\pgh - S$. It implies that the graph $\pgh - S$ is disconnected and so $\kappa(\pgh) \leq |S|$, a contradiction. \\
     Thus, either $G$ is cyclic or $G$ does not contain any element of order greater than two. Later part implies that $G \cong \mathbb{Z}_{2} \times \mathbb{Z}_{2} \times \cdots \times \mathbb{Z}_{2}.$

    Conversely, If $G$ is cyclic, by Theorem 2.4 of \cite{a.Bera2017}, the enhanced power graph $\pgh $ is minimally connected. If $G \cong \mathbb{Z}_{2} \times \mathbb{Z}_{2} \times \cdots \times \mathbb{Z}_{2}$, then $\pgh$ is a star graph and so $\pgh$ is minimally connected.
\end{proof}

In what follows, we characterize finite groups whose order superpower graphs are minimally (edge) connected. First note that, 
if the group $G$ is of full exponent, then the order superpower graph $\sg$ has more than two dominating vertices. Also, it is well known that the order superpower graph of the group $G$ is complete if and only if $G$ is a $p$-group. Consequently, we obtain the following corollary.
\begin{corollary}{\label{sg full exponent}}
    Let $G$ be a finite group of full exponent. Then the order superpower graph $\sg$ is minimally edge connected if and only if $G$ is a $p$-group.
\end{corollary}

The following proposition describes minimal edge connectedness of the order super power graphs of the groups of even order.
\begin{proposition}
    Let $G$ be a group of even order which is not a $p$-group. Then the order superpower graph $\sg$ is not minimally edge connected.
\end{proposition}
\begin{proof}
   Let $G$ be a group of even order which is not a $p$-group. Clearly, the identity element $e\in G$ is a dominating vertex of $\sg$. Let $x,y \in G$ such that $o(x)=2$ and $o(y) = p$, where $p$ is an odd prime. Then $x$ is adjacent to all other elements of order 2 in $G$. Note that $G$ has odd number of elements of order 2. For $y_1 \in G$, if $x \sim y_1$ such that $o(y_1) \geq 2$, then $x \sim y_1^{-1}$. Consequently, $deg(x)$ is even in $\mathcal{S}^*(G)$. Now, if $y \sim z$ for $z \neq e \in G$, then $y \sim z^{-1}$. Also $y \sim y^{-1}$. Thus $deg(y)$ is odd in $\mathcal{S}^*(G)$. It follows that $\mathcal{S}^*(G)$ is not regular. By Theorem \ref{graph}, $\sg$ is not minimally edge connected. 
   \end{proof}
\begin{corollary}
  \begin{enumerate}
      \item [(i)] The order superpower graph $\sg$ of a non-cyclic simple group cannot be minimally edge connected.
      \item [(ii)] If $G \in \{ S_n, D_{2n}, SD_{8n}, Q_{4n}\}$, then the order superpower graph $\sg$ is not minimally edge connected.
  \end{enumerate}
\end{corollary}

\begin{theorem}\label{sg-minimallyconnected}
    Let $G$ be a finite group. Then the order superpower graph $\sg$ is minimally connected if and only if $G$ is a $p$-group.
\end{theorem}
\begin{proof} 
    Let $\sg$ be minimally connected. If all the non-identity elements of $G$ are of order two, then $G$ is a $2$-group. If there exits an element $x \in G$ such that $o(x) \geq 3$, then in the similar lines of the proof of the Theorem \ref{pgh minimally connected}, we obtain that $G$ is $p$-group.

    Conversely, if $G$ is a $p$-group then $\sg$ is complete and so is minimally connected.
    \end{proof}

Now we classify all the nilpotent groups such that the minimum degree and the vertex connectivity of its order superpower graphs are equal. The following proposition is useful.

\begin{proposition}{\label{order 2 unique}}
    Let $G$ be a finite group which is not a $p$-group and let $x \in G$ such that $deg(x) = \delta(\sg) = \kappa(\sg)$. Then the following hold:
    \begin{enumerate}
        \item[(i)] $x$ is the unique element of order $2$ in $G$.
        \item[(ii)] There does not exist any element of order $2^{\alpha}$, where $\alpha \geq 2$, in $G$.
    \end{enumerate}
\end{proposition}

\begin{proof}
(i) Assume that $o(x) \geq 3$. Clearly, $N[x] \neq G$. Let $S = N(x) \setminus \{x\}$. Note that $S$ is a vertex cut-set of $\sg$ and $deg(x) > |S| \geq \kappa'(\sg)$, a contradiction. Thus $o(x)=2$. Now suppose that $G$ has $2$ elements $x,y \in G$ such that $o(x)=o(y)=2$. Then $deg(x)=deg(y)= \delta(\sg)$. Observe that $S' = N(x)\setminus \{y\}$ is a vertex cut-set of $\sg$ and $|S'| < deg(x) = \delta(\sg)$, again a contradiction.\\

(ii) By part (i), there exists a unique element $x$ of order $2$ such that $deg(x) = \delta(\sg)$. If possible, let $y \in G$ such that $o(y) = 2^\alpha$ for $\alpha \geq 2$. Consider $y_1 \in \langle y \rangle$ such that $o(y_1)=4$. Note that for $z \neq x \in G$, if $z \sim y_1$, then $z \sim x$. Consequently, $deg(y_1) \leq deg(x)$, which is not possible.
\end{proof}

\begin{theorem}\label{min degree = vertex connectivity of SG}
    Let $G$ be a finite nilpotent group. Then $\delta(\sg) = \kappa(\sg)$ if and only if one of the following holds:
    \begin{enumerate}
        \item[(i)] $G$ is  a $p$-group.
        \item[(ii)] $G \cong \mathbb{Z}_{2} \times P$ where $P$ is a  $p-$group with $p > 2$.
         \end{enumerate}
\end{theorem}
\begin{proof}
Let $G = P_1 \times P_2 \times \cdots \times P_r$ be a finite nilpotent group such that $\delta(\sg)=\kappa(\sg)$. If $G$ is a $p$-group, then the result holds. Now suppose $r\geq2$. If possible, assume that $r\geq3$. By Proposition \ref{order 2 unique}, $G$ has Sylow 2-subgroup, which is isomorphic to a cyclic group of order two. Without loss of generality, assume that $P_1\cong \mathbb{Z}_2$ and $P_2$ be a Sylow $p_2$-subgroup of $G$ such that $|P_2|<|P_j|$ for all $j \in [r]\setminus\{1,2\}$. Let $x \in G$ be such that $o(x)=2$. Then by Proposition \ref{order 2 unique}, we have $deg(x) = \delta(\sg)$. Consider $y \in G$ such that $o(y)=2p_2$. Note that for $z\neq x$, if $z \sim y$ and $o(x) \neq p_2$, then $z \sim x$. Also, $x$ is adjacent to all the elements whose order is of the form $2p_{3}^{k}$. Consequently, we get $deg(x) \geq deg(y)-|S_1|+|S_2|$, where $S_1 = \{ x \in G: o(x)= p_2\}$ and $S
_2 = \{ y \in G: o(y)= 2p_3^{\alpha},$ where $\alpha \geq 1\}$. Note that $|S_1| \leq |P_2|$ and $|S_2|=|P_3|$ since $|P_2| < |P_3|$. Therefore, $deg(x) > deg(y)$, which is a contradiction. Thus, $r=2$ and so $G \cong \mathbb{Z}_2 \times P$, where $P$ is a $p$-group.

Conversely, if $G$ is a $p$-group, then $\sg$ is complete and hence $\delta(\sg)=\kappa(\sg)$. Now, let $G \cong \mathbb{Z}_2 \times P$, where $P$ is a $p$-group with exponent $p^{\alpha}$. Thus, $\pi_{G} = \{ 1,2,p,p^2,\ldots,p^{\alpha},2p,2p^{2},\ldots,2p^{\alpha}\}$. Suppose $x \in G$ such that $o(x)=2$. For $1\leq i\leq \alpha$, consider the sets $S_{i} = \{ g \in G: o(g)=p^i\}$ and $S_{i}^{'}=\{g \in G: o(g)=2p^i\}$. Note that $G = \{x,e\} \cup\{ \bigcup_{i=1}^{\alpha} S_i\} \cup \{ \bigcup_{i=1}^{\alpha}S_i^{'}\}$ is a partition of $G$. Clearly, $|S_i|=|S_{i}^{'}|$ and $\sum_{i=1}^{\alpha}|S_i| = |P|-1$. Note that $N(x)=\{e\} \cup ( \bigcup_{i=1}^{\alpha}S_i^{'})$. Consequently, $deg(x)=|P|$. Now we show that $\delta(\sg)=|P|$. Clearly, $deg(e)>deg(x)$. Suppose $y \neq x \in G$ is a non-identity element. Now, if $y \in S_i$ for some $i$, then note that $(\bigcup_{j=1}^{\alpha}S_j) \cup S_{i}^{'} \cup \{e\} \subseteq N(y)$. It follows that $deg(y)>deg(x)$. Similarly, for $y \in S_{i}^{'}$ for some $i$, we obtain $deg(y) > deg(x)$. It implies that $deg(x)= \delta(\sg) =|P|$. Now, we prove that $\kappa(\sg) = |P|$. On the contrary, assume that $\kappa(\sg) < |P|$. Let $T$ be a cut-set of minimum size in $\sg$. Clearly, $e \in T$. Suppose $u_1$ and $u_2$ are two elements such that there exists no path between $u_1$ and $u_2$ in $\sg \setminus T$. Then we have the following cases:\\
    \textbf{Case 1:} \emph{$o(u_1) = 2$ and $o(u_2) = p_2^{k}$, for some $1 \leq k \leq \alpha$}. Now, for $v \in \bigcup_{i=k}^{\alpha} S_{i}^{'}$, we have $u_{1} \sim v \sim u_2$ in $\sg$. Thus. $\bigcup_{i=k}^{\alpha} S_{i}^{'} \subseteq T$. For $1 \leq j \leq k-1$, for each $v_1 \in S_j$ and $v_2 \in S_{j}^{'}$, there exist a path $u_1 \sim v_2 \sim v_1 \sim u_2$ in $\sg$. Therefore, for each $j, 1 \leq j \leq k-1$, either $S_{j}^{'} \subseteq T$ or $ S_j \subseteq T$. Thus, we get $|T| \geq |P|$, a contradiction.\\
    \textbf{Case 2:} \emph{$o(u_1) = p^{\delta}$ and $o(u_2) = 2p^{\delta^{'}}$ where $\delta > \delta^{'}$}. For $\delta \leq i \leq \alpha$, if $z \in S_{i}^{'}$, then there exist a path $u_1 \sim z \sim u_2$. Thus $\bigcup_{i= \delta}^{\alpha}S_{i}^{'} \subseteq T$. Now, for $\delta^{'} + 1 \leq i \leq \delta - 1$, let $z_1 \in S_i$ and $z_2 \in S_{i}^{'}$, we get a path $u_1 \sim z_1 \sim z_2 \sim u_2$. It implies that either  $\bigcup_{\delta^{'}+1} ^{\delta-1} S_{i} \subseteq T$ or $\bigcup_{\delta^{'}+1} ^{\delta-1} S_{i}^{'} \subseteq T$. Now, for $1 \leq i \leq \delta^{'}$ and $t \in S_{i}$ then we obtain a path $u_1 \sim t \sim u_2$. It implies that $\bigcup_{i=1}^{\delta^{'}}S _i \subseteq T$. Thus, $|T| \geq |P|$; which is a contradiction. Hence, $\kappa(\sg) = |P|$.
 \end{proof}
\section*{Declarations}

\textbf{Funding}: The second and the third author wishes to acknowledge the support of Core Research Grant (CRG/2022/001142) funded by  SERB.

\textbf{Conflicts of interest/Competing interests}: There is no conflict of interest regarding the publishing of this paper.

\textbf{Availability of data and material (data transparency)}: Not applicable.

\vspace{.3cm}
\textbf{Code availability (software application or custom code)}: Not applicable.

\vspace{1cm}
\noindent
{\bf Parveen\textsuperscript{\normalfont 1}, \bf Manisha\textsuperscript{\normalfont 1}, \bf Jitender Kumar\textsuperscript{\normalfont 1}}
\bigskip

\noindent{\bf Addresses}:

\vspace{5pt}

\end{document}